\documentclass[11pt,a4paper]{article}
\usepackage{amsmath}

\usepackage{epsf,epsfig,amsfonts,amsgen,amsmath,amstext,amsbsy,amsopn,amsthm
}
\usepackage{ebezier,eepic}
\usepackage{color}
\usepackage{multirow}
\newtheorem{thm}{Theorem}
\newtheorem{theoremA}{Theorem}

\newtheorem{lem}{Lemma}
\newtheorem{false statement}{False statement}

\theoremstyle{definition}

\newtheorem{problem}{Problem}

\newcounter{mathitem}
  {\begin{list}{{$(\roman{mathitem})$}}{
   \setcounter{mathitem}{0}
   \usecounter{mathitem}
   \setlength{\topsep}{0pt plus 2pt minus 0pt}
   \setlength{\parskip}{0pt plus 2pt minus 0pt}
   \setlength{\partopsep}{0pt plus 2pt minus 0pt}
   \setlength{\parsep}{0pt plus 2pt minus 0pt}
   \setlength{\leftmargin}{35pt}
   \setlength{\itemsep}{0pt plus 2pt minus 0pt}}}
  {\end{list}}

\begin{document}

\title{\bf Eigenvalues and cycles of consecutive lengths}

\date{}

\author{Binlong Li\thanks{School of Mathematics and Statistics,
Northwestern Polytechnical University, Xi'an, Shaanxi 710072,
P.R.~China. Email: binlongli@nwpu.edu.cn. Partially supported
by the NSFC grant (No.\ 12171393) and Natural Science Foundation of Shaanxi Provincial Department of Education (2021JM-040).}~~~ Bo
Ning\thanks{Corresponding author. College of Computer Science, Nankai University, Tianjin 300350, P.R.
China. Email: bo.ning@nankai.edu.cn. Partially supported
by the NSFC grant (No.\ 11971346).}}
\maketitle

\begin{center}
\begin{minipage}{140mm}
\small\noindent{\bf Abstract:}
As the counterpart of classical theorems on cycles of consecutive lengths
due to Bondy and Bollob\'as in spectral graph theory, Nikiforov
proposed the following open problem in 2008:
What is the maximum $C$ such that for all positive $\varepsilon<C$ and sufficiently large $n$, every
graph $G$ of order $n$ with spectral radius $\rho(G)>\sqrt{\lfloor\frac{n^2}{4}\rfloor}$
contains a cycle of length $\ell$ for each integer $\ell\in[3,(C-\varepsilon)n]$.
We prove that $C\geq\frac{1}{4}$, improving the existing bounds.
Besides several novel ideas, our proof technique is partly inspired by the recent research on Ramsey numbers of star versus
large even cycles due to Allen, {\L}uczak, Polcyn and Zhang,
and with aid of a powerful spectral inequality.
We also derive an Erd\H{o}s-Gallai-type edge number condition for
even cycles, which may be of independent
interest.


\end{minipage}
\end{center}


This note is devoted to an open problem on cycles with consecutive lengths
due to Nikiforov \cite{N08}. Throughout this note, for a graph $G$, $V(G)$
denotes the vertex set of $G$, $|G|$ and $e(G)$ denotes the number of
vertices and edges in $G$, respectively. We denote by $\rho(G)$
the spectral radius of $G$. Let $G_1$ and $G_2$ be two graphs.
We use $G_1\vee G_2$ to denote the new graph obtained from the disjoint
union of $G_1$ and $G_2$ by adding all edges $xy$ where $x\in V(G_1)$
and $y\in V(G_2)$.

A classical theorem due to Bondy \cite{B71} says that every
hamiltonian graph $G$ on $n$ vertices contains cycles of all lengths
$\ell \in [3,n]$ if $e(G)\geq \frac{n^2}{4}$, unless $n$ is even
and $G$ is isomorphic to $K_{\frac{n}{2},\frac{n}{2}}$.
If one drops the condition that ``$G$ is hamiltonian" in Bondy's theorem, a result in Bollob\'{a}s'
textbook \cite[Corrolary~5.4]{B76} states that such a graph
contains all cycles $C_{\ell}$ for
each $\ell \in [3,\left\lfloor\frac{n+3}{2}\right\rfloor]$.
Nikiforov \cite{N08} considered cycles of consecutive lengths
from a spectral perspective and posed the following open problem.

\begin{problem}[Nikiforov \cite{N08}]\label{Prob:2}
What is the maximum $C$ such that for all positive $\varepsilon<C$ and sufficiently large $n$, every
graph $G$ of order $n$ with $\rho(G)>\sqrt{\lfloor\frac{n^2}{4}\rfloor}$
contains a cycle of length $\ell$ for every integer $3\leq \ell\leq (C-\varepsilon)n$.
\end{problem}
One may guess $C=\frac{1}{2}$. However, the class of graphs
$G=K_s\vee (n-s)K_1$ where $s=\left\lfloor\frac{(3-\sqrt{5})n}{4}\right\rfloor$ (see \cite{N08})
shows $C\leq \frac{(3-\sqrt{5})}{2}$. Nikiforov \cite{N08} proved that
$C\geq \frac{1}{320}$. Ning and Peng \cite{NP20} slightly refined this
as $C\geq \frac{1}{160}$. Only very recently, Zhai and
Lin \cite{ZL} improved these results to $C\geq\frac{1}{7}$.

The purpose of this note is to show $C\geq \frac{1}{4}$
by completely different methods.

One key ingredient is from Ramsey Theory.
Very recently, Allen, {\L}uczak, Polcyn and Zhang \cite{ALPZ20+} studied the
Ramsey numbers of large even cycles versus stars. Their key technique of finding large
even cycles is to combine a theorem on cycle lengths with given parity under minimum degree
condition due to Voss and Zuluaga \cite{VZ77} and a theorem on the distribution
of cycles with consecutive lengths due to Gould, Haxell, and Scott \cite{GHS02}.
The proof of our main theorem
is somewhat inspired by this idea.

\begin{thm}\label{Thm:SpectraConsecutiveCycles}\footnote{If $0<\varepsilon<10^{-6}$,
then we can choose $N=2.5\times 10^{10}{\varepsilon}^{-1}$.}
Let $\varepsilon$ be real with $0<\varepsilon<\frac{1}{4}$. Then there
exists an integer $N:=N(\varepsilon)$, such that if $G$ is a graph on
$n$ vertices with $n\geq N$ and $\rho(G)>\sqrt{\lfloor\frac{n^2}{4}\rfloor}$,
then $G$ contains all cycles $C_{\ell}$ with $\ell \in [3,(\frac{1}{4}-\varepsilon)n]$.
\end{thm}

Before the proof, we first collect various known results on cycles that will be used in our arguments.
We use $ec(G)$ and $oc(G)$ to denote the length of a longest even cycle
and the length of a longest odd cycle, of $G$, respectively.

\begin{theoremA}[Voss and Zuluaga \cite{VZ77}]\label{Thm:Voss-Zuluage}
(1) Every 2-connected graph $G$ with $\delta(G)\geq k\geq 3$ having at least $2k+1$
vertices satisfies that $ec(G)\geq 2k$.
(2) Every 2-connected non-bipartite graph $G$ with $\delta(G)\geq k\geq 3$ having at least $2k+1$
vertices satisfies that $oc(G)\geq 2k-1$.
\end{theoremA}

\begin{theoremA}[Ore \cite{O61}]\label{Thm:Ore}
Let $G$ be a graph on $n$ vertices. If $G$ contains no Hamilton cycle, then
$e(G)\leq \binom{n-1}{2}+1$.
\end{theoremA}

\begin{theoremA}[Gould, Haxell and Scott \cite{GHS02}]\footnote{In \cite{GHS02},
the authors set $K=\frac{7.5\times 10^5}{c^5}$.}\label{Thm:GHS}
For every real number $c>0$, there exists a constant $K:=K(c)$
such that the following holds. Let $G$ be a graph with $n\geq \frac{45K}{c^4}$ vertices
and minimum degree at least $cn$. Then $G$ contains a cycle of length $t$
for every even $t\in [4,ec(G)-K]$ and every odd $t\in [K,oc(G)-K]$.
\end{theoremA}

Another tools are the following two spectral inequalities, of which the second one was originally
conjectured by Guo et al. \cite{GWL19}.

\begin{theoremA}[Hong \cite{H93}]\label{Thm:Hong}
Let $G$ be a graph on $n$ vertices and $m$ edges. If $\delta(G)\geq1$ then
$\rho(G)\leq \sqrt{2m-n+1}$.
\end{theoremA}

\begin{theoremA}[Sun and Das \cite{SD20}]\label{Thm:Sun-Das}
Let $G$ be a graph with minimum degree $\delta(G)\geq 1$. For any
$v\in V(G)$, we have $\rho^2(G-v)\geq\rho^2(G)-2d(v)+1$.
\end{theoremA}

\begin{theoremA}[Nosal \cite{N70}]\label{Thm:Nosal}
Let $G$ be a graph on $m$ edges. If $\rho(G)>\sqrt{m}$,
then $G$ contains a triangle.
\end{theoremA}

Our proof of
Theorem \ref{Thm:SpectraConsecutiveCycles} uses the following three lemmas.
A graph is called a \emph{theta-graph} if it consists of three paths starting and
ending with two same vertices and internally vertex-disjoint. The first
lemma is very basic.
\begin{lem}\label{Lem:Theta}
Let $G$ be a graph containing no theta-graphs. Then each block of $G$
is an isolated vertex, or an edge (a $K_2$), or a cycle.
\end{lem}

\begin{lem}\label{Lem:Erdos-Gall-Voss}
Let $G$ be a graph on $n$ vertices.
If $ec(G)\leq 2k$ where $k\geq 1$ is an integer,
then $e(G)\leq \frac{(2k+1)(n-1)}{2}$.
\end{lem}

\begin{proof}
If $n\leq 2k+1$, then $e(G)\leq \binom{n}{2}\leq \frac{(2k+1)(n-1)}{2}$.
If $n=2k+2$, then by Theorem \ref{Thm:Ore}, we have $e(G)\leq \binom{2k+1}{2}+1\leq \frac{(2k+1)(n-1)}{2}$.
Next, we assume $n\geq 2k+3$.

Let $k=1$. We shall prove that if a graph on $n$ vertices contains
no even cycles then $e(G)\leq \frac{3(n-1)}{2}$.
By Lemma \ref{Lem:Theta}, every block of $G$ is an isolated vertex, an edge or an odd cycle. Let $c$
be the number of blocks which are odd cycles. By induction on $c$,
we can prove that
$e(G)\leq n+c-1\leq n-1+\frac{n-1}{2}=\frac{3(n-1)}{2}$.
In the following, we suppose $k\geq 2$.

Let $v\in V(G)$ with $d_G(v)=\delta(G)$, and $G':=G-v$. Note that $G'$ satisfies that
$|G'|\geq2k+2$ and $G'$ contains no even cycle of length more than $2k$.
By induction hypothesis, if $d(v)\leq k$, then we have
$e(G)=e(G')+\delta\leq \frac{(2k+1)(n-2)}{2}+k<\frac{(2k+1)(n-1)}{2}$,
as required. Thus, $\delta(G)\geq k+1\geq 3$. If $G$ is 2-connected,
then by Theorem \ref{Thm:Voss-Zuluage},  $G$ contains an even cycle
of length at least $2k+2$, a contradiction. Thus, $G$ contains a cut-vertex
or is disconnected. In such a case, let $G=G_1\cup G_2$,
where $|V(G_1)\cap V(G_2)|\leq 1$, and we use induction to $G_1$
and $G_2$ to get the required result.
The proof is complete.
\end{proof}

\begin{lem}\label{Lem:Sun-Das}
Let $G$ be a graph. For any $v\in V(G)$, we have
$\rho^2(G)\leq \rho^2(G-v)+2d_G(v).$
\end{lem}
\begin{proof}
If $\delta(G)\geq 1$, then by Theorem \ref{Thm:Sun-Das}, the result follows.
Suppose that $\delta(G)=0$. If $d_G(v)=0$, then $\rho^2(G)=\rho^2(G-v)\leq \rho^2(G-v)+2d_G(v)$.
Now assume that $d_G(v)\geq 1$. Let $I$ be the set of isolated vertices in $G$ and $G':=G-I$.
Then $\delta(G')\geq 1$. By Theorem \ref{Thm:Sun-Das},
$\rho^2(G)=\rho^2(G')\leq \rho^2(G'-v)+2d_{G'}(v)=\rho^2(G-v)+2d_{G}(v)$.
This proves Lemma \ref{Lem:Sun-Das}.
\end{proof}

Now we give the proof of Theorem \ref{Thm:SpectraConsecutiveCycles}.

\noindent
{\bf Proof of Theorem \ref{Thm:SpectraConsecutiveCycles}.}
If $G$ is disconnected, then we use $\omega(G)$
to denote the number of components in $G$.
Then we can add $\omega(G)-1$ edges to
make it connected. Note that $\rho(G')>\rho(G)$.
For any integer $k\geq 3$, $G$ contains a $C_{\ell}$
if and only if $G'$ contains a  $C_{\ell}$.
Thus, we can assume $G$ is connected.

By Theorem \ref{Thm:Hong}, we have
$$\frac{n^2-1}{4}\leq \rho^2(G)\leq
2m-n+1.$$ One can derive that
$2m\geq\frac{n^2+4n-5}{4}$. Thus, the average degree
$d(G):=\frac{2m}{n}>\frac{n}{4}$.

Let $H$ be a subgraph of $G$ defined by a sequence of graphs
$G_0,G_1,\ldots,G_k$ such that:\\
(1) $G=G_0$ and $H=G_k$;\\
(2) for every $i\in[0,k-1]$, there is $v_i\in V(G_i)$ such that
$d_{G_i}(v_i)\leq\frac{n}{8}$ and $G_{i+1}=G_i-v_i$;\\
(3) for every $v\in V(G_k)$, $d_{G_k}(v)>\frac{n}{8}$.\\
We claim that $d(H)>\frac{n}{4}$. Suppose not the case. Then there
is a least integer $i\in[1,k]$ with $d(G_i)\leq\frac{n}{4}$. This implies
that
$$d(G_{i-1})=\frac{2d_{G_{i-1}}(v_{i-1})+|G_i|d(G_i)}{|G_i|+1}\leq\frac{n}{4},$$
a contradiction. Thus, we conclude that $d(H)>\frac{n}{4}$ and
$\delta(H)>\frac{n}{8}$.

\underline{Case A: Even cycle.} Note that
$e(H)=\frac{d(H)|H|}{2}>\frac{\frac{n}{4}(|H|-1)}{2}$. By Lemma
\ref{Lem:Erdos-Gall-Voss}, $ec(H)>\left\lfloor\frac{n}{4}\right\rfloor\geq \frac{n}{4}-1$. Recall that
$\delta(H)>\frac{n}{8}$. By Theorem \ref{Thm:GHS}, $H$ contains all
even cycles $C_{\ell}$ with $\ell \in [4,ec(H)-K]$ if
$|H|\geq45\cdot8^4\cdot K$, where $K=K(\frac{1}{8})=\frac{7.5\times 10^{5}}{(\frac{1}{8})^5}$
corresponds to the constant $K(c)$ in
Theorem \ref{Thm:GHS}. Clearly $|H|>\frac{n}{4}$. Let $n_1$ be an
integer satisfying
$$({\rm i})\ \frac{n_1}{4}\geq 45\cdot8^4\cdot K;\  ({\rm ii})\ \varepsilon n_1\geq K+1.$$
Now if $n\geq n_1$ then $G$ contains $C_{\ell}$ with every even
$\ell\in[4,(\frac{1}{4}-\varepsilon)n]$.

\underline{Case B: Odd cycle.} Set $h=|H|$. By Lemma
\ref{Lem:Sun-Das}, we have
$$
\rho^2(G)\leq\rho^2(H)+2\sum_{i=0}^{k-1}d_{G_i}(v_i)\leq\rho^2(H)+2k\cdot\frac{n}{8}=\rho^2(H)+\frac{kn}{4},
$$
where $G_i$, $v_i$ are those in the definition of $H$, and $k=n-h$.
This implies that $\rho(H)\geq\frac{\sqrt{nh-1}}{2}$. Since $h\leq n$ and
$\rho(H)>\frac{\sqrt{nh-1}}{2}\geq\sqrt{\lfloor\frac{h^2}{4}\rfloor}$, by Theorem \ref{Thm:Nosal}
and Mantel's theorem, $H$ contains a triangle, and so is
non-bipartite.

Let $F$ be a subgraph of $H$ defined by a sequence of graphs
$H_0,H_1,\ldots,H_s$ such that:\\
(1) $H_0=H$ and $H_s=F$;\\
(2) for every $i\in[0,s-1]$, there is a cut-vertex $v_i$ of $H_i$ and $H_{i+1}:=H_i-v_i$;\\
(3) $H_s$ has no cut-vertex.\\
Note that $w(H_{i+1})\geq w(H_i)+1$. Clearly
$w(H)\leq 8$, for otherwise $H$ will have a vertex of degree less
than $\frac{n}{8}$. We claim that $w(F)\leq 8$. Suppose to the
contrary that there is a least integer $i\in [0,s]$ such that $w(H_i)\geq 9$. Notice
that for $i\leq 8$, $\delta(H_i)>\frac{n}{8}-8$. As
$w(H_i)\geq 9$, $H_i$ has a vertex with degree less that
$\frac{|H_i|}{9}<\frac{n}{9}$, a contradiction when $n\geq577$. Thus
we conclude that $w(F)\leq 8$, and specially, $|V(F)|\geq h-7$.

By Lemma \ref{Lem:Sun-Das}, we have
$$
\rho^2(H)\leq\rho^2(F)+2\sum_{i=0}^{s-1}d_{H_i}(v_i)\leq\rho^2(F)+2s(h-1)\leq\rho^2(F)+14(h-1).
$$
Since $d(H)>\frac{n}{4}$, we obtain $h>\frac{n}{4}+1$. Thus,
\begin{equation*}
\begin{split}
\rho(F) & \geq\sqrt{\rho^2(H)-14(h-1)}\geq\sqrt{\frac{nh-1}{4}-14(h-1)}\\
        & =\sqrt{\left(\frac{n}{4}-14\right)h+\frac{55}{4}}>\sqrt{\left(\frac{n}{4}-14\right)\left(\frac{n}{4}+1\right)+\frac{55}{4}}\geq\frac{n}{4}-7\\
\end{split}
\end{equation*}
when $n\geq197$.

Recall that $F$ has no cut-vertex and $\delta(F)\geq \delta(H)-7\geq \frac{n}{8}-7\geq 2$, i.e., every component of $F$ is
2-connected. Let $F_1$ be a component of $F$ with
$\rho(F_1)=\rho(F)$. Thus we have $\delta(F_1)\geq\frac{n}{8}-7$ and
$\rho(F_1)>\frac{n}{4}-7$. Specially $|F_1|>\frac{n}{4}-6$.

We claim that $\delta(F_1)\geq\frac{|F_1|}{8}$. Recall that
$\delta(H)\geq\frac{n}{8}\geq\frac{|H|}{8}$, we assume that $F_1\neq
H$. This implies that $F$ has a second component, say $F_2$. Since
$\delta(F)\geq\frac{n}{8}-s$, we have $|F_2|\geq\frac{n}{8}-s+1$
(here $s$ is the subscript of $F=H_s$). This implies that
$|F_1|\leq h-s-(\frac{n}{8}-s+1)<\frac{7h}{8}$. Thus
$\delta(F_1)\geq\frac{n}{8}-7\geq\frac{7n/8}{8}\geq\frac{|F_1|}{8}$
when $n\geq 448$.

Now we show that $F_1$ is non-bipartite. Recall that $H$ is
non-bipartite. So we assume that $F_1\neq H$. By the analysis above
we have $|F_1|<\frac{7h}{8}$. Thus
$$\rho^2(F_1)=\rho^2(F)\geq\left(\frac{n}{4}-14\right)h+\frac{55}{4}\geq\left(\frac{h}{4}-14\right)h+\frac{55}{4}>\frac{(7h/8)^2}{4}>\frac{|F_1|^2}{4}$$
when $h\geq 238$. Recall that $h>\frac{n}{4}+1$, we have that $F$ is
non-bipartite when $n\geq 944$.

By Theorem \ref{Thm:Voss-Zuluage}, $oc(F_1)\geq
\min\{2\delta(F_1)-1,|F_1|\}\geq \min\{\frac{n}{4}-15,\rho(F_1)\}=\frac{n}{4}-15$. By Theorem \ref{Thm:GHS}, $F_1$
contains all odd cycles $C_{\ell}$ for $\ell\in[K,\frac{n}{4}-15-K]$
if $|F_1|\geq 45\cdot8^4\cdot K$, where $K=K(\frac{1}{8})$ is the
constant as in Theorem \ref{Thm:GHS}. A theorem of Nikiforov
\cite[Theorem~1]{N08}\footnote{By refining the proof, one can let
$N=8400$.} states that there exists a sufficiently large $N$ such
that any graph of order $n\geq N$ such that  $\rho(G)>\sqrt{\lfloor\frac{n^2}{4}\rfloor}$
has a cycle of length $\ell$ for
every $\ell \in[3,\frac{n}{320}]$. Let $n_2$ be an integer such that
$$({\rm i})\ \frac{n_2}{4}-6\geq 45\cdot8^4\cdot K;\ ({\rm ii})\ n_2\geq\max\{944,N\};\ ({\rm iii})\ \frac{n_2}{320}\geq K;\ ({\rm iv})\
\varepsilon n_2\geq K+15.$$ Now if $n\geq\max\{n_1, n_2\}$, then $G$
contains all cycles $C_{\ell}$ with $\ell
\in[3,(\frac{1}{4}-\varepsilon)n]$.
The proof is complete. \hfill{\rule{4pt}{8pt}}

\bigskip

\noindent
{\bf{Note added}}\\

The result presented in this note was a part of work in \cite{LN} due to the current
authors. Till now, the problem of determining the value of $C$ is still open.

After the first version of this paper was written and submitted for publication, we have learned
our paper was cited in \cite{ZZ}, and the method
developed here (together with some powerful technique developed by Nikiforov in \cite{N08}) has been used by
Zhang and Zhao \cite{ZZ} to obtain a spectral condition for the existence of cycles with consecutive
odd lengths in non-bipartite graphs.

\section*{Acknowledgment}
The authors express their gratitude to the anonymous
reviewers for their detailed and constructive comments which
are very helpful to the improvement of the presentation of this
note, and to Prof. Agelos Georgakopoulos, the Managing Editor, for his
excellent editorial job.
The second author is also grateful to Xing Peng who gave an online talk on Ramsey Theory
for Tsinghua University. After Xing's talk, the second author drew his attention to the topic of Ramsey Theory of books versus
even cycles and realized the technique in \cite{ALPZ20+}.

\end{document}